\documentclass[12pt]{article}
\usepackage{amsthm,amsmath,amssymb,latexsym,amscd}

\renewcommand{\a}{\mathbf{a}}

\newcommand{\e}{\epsilon}
\newcommand{\x}{\mathbf{x}}
\newcommand{\y}{\mathbf{y}}

\newcommand{\g}{\frak{g}}
\renewcommand{\L}{\mathcal{L}}

\newcommand{\Hom}{\mathrm{Hom}}

\newcommand{\Coder}{\mathrm{Coder}}
\newcommand{\Der}{\mathrm{Der}}

\newcommand{\p}{\prime}
\renewcommand{\c}{\circ}
\newcommand{\ot}{\otimes}
\newcommand{\pa}{\partial}
\newcommand{\what}{\widehat}

\newtheorem{definition}{Definition}[section]

\newtheorem{lemma}[definition]{Lemma}
\newtheorem{proposition}[definition]{Proposition}

\newtheorem{theorem}[definition]{Theorem}
\newtheorem{corollary}[definition]{Corollary}
\newtheorem{remark}[definition]{Remark}
\newtheorem{example}[definition]{Example}
\date{}

\begin{document}

\title{
Higher derived brackets,
strong homotopy associative algebras and Loday pairs
}
\author{K. UCHINO}
\maketitle
\footnote{Mathematics Subject Classifications (2000):
17A32, 53D17}
\footnote{Keywords: strong homotopy associative algebras,
higher derived brackets, Leibniz pairs}
\abstract
{
We give a quick method of constructing strong homotopy
associative algebra, namely,
the higher derived product construction.
This method is associative analogue of classical
higher derived bracket construction in the category
of Loday algebras.
We introduce a new type of algebra, {\em Loday pair},
which is noncommutative analogue of classical Leibniz pair.
We study strong homotopy Loday pairs and
the higher derived brackets on the Loday pairs.
}
%%%%%%%%%%%%%%%%%%%%%%%
\section{Introduction}
%%%%%%%%%%%%%%%%%%%%%%%

Let $(\g,d,[,])$ be a differential graded
(dg, for short) Lie algebra.
We define a new product by $[x,y]_{d}:=(-1)^{|x|}[dx,y]$.
This product is called a \textbf{derived bracket}
of Koszul-Kosmann-Schwarzbach (\cite{K}).
It is known that the algebra of the derived bracket
is a Loday algebra (so-called Leibniz algebra),
i.e., $[x,y]_{d}$ satisfies the Leibniz identity:
$$
[x,[y,z]_{d}]_{d}=[[x,y]_{d},z]_{d}+(-1)^{|x||y|}[y,[x,z]_{d}]_{d}.
$$
The derived bracket construction is a method of
constructing new algebra structure.
It plays important roles in modern analytical
mechanics and in differential geometry.
For instance, a Poisson bracket on a smooth manifold
is a derived bracket of a graded Poisson bracket
which is called a Schouten-Nijenhuis bracket.
$$
\{f,g\}=(-1)^{f}[[\pi,f]_{SN},g]_{SN},
$$
where $f,g$ are smooth functions,
$\pi$ is a Poisson structure tensor,
$[,]_{SN}$ is the Schouten-Nijenhuis bracket
and $\{,\}$ is the induced Poisson bracket.
Since $d_{\pi}:=[\pi,-]_{SN}$ is a differential,
the Poisson bracket is a derived bracket.
We recall another example of derived brackets.
Let $f=f(p,q)$ and $g=g(p,q)$ be super functions
on a super symplectic-manifold,
where $(p,q)$ is a canonical cordinate of the manifold.
We consider a Laplacian with odd degree
$\Delta_{BV}:=\sum(\pm)\frac{\pa^{2}}{\pa p\pa q}$
and define a differential $d_{BV}:=[\Delta_{BV},-]$.
It is known that
the derived bracket associated with $d_{BV}$
is a Poisson bracket (so-called BV-bracket):
$$
(f,g)=
\sum\frac{f\overleftarrow{\pa}}{\pa p}
\frac{\overrightarrow{\pa}g}{\pa q}-
\frac{f\overleftarrow{\pa}}{\pa q}
\frac{\overrightarrow{\pa}g}{\pa p}
=(\pm)[d_{BV}\hat{f},\hat{g}],
$$
where $\hat{f}(-):=f\times(-)$ is a scalar multiplier,
$(\pm)$ is an appropriate sign
and $[,]$ is a Lie bracket (commutator).
Thus various bracket products (Lie algebroid
brackets, Courant brackets, BV-brackets and so on)
are given as derived brackets (see \cite{K2}).\\
\indent
The idea of the derived bracket arises
in several mathematical areas.
We recall a derived bracket
in the category of associative algebras.
Let $(A,*,d)$ be a dg associative algebra.
We define a modified product by
$a*_{d}b:=(-1)^{|a|}(da)*b$, $a,b\in A$.
Then it is again associative.
This new product is called a \textbf{derived product},
which is used in the study
of Loday type algebras (cf. Loday \cite{Lod}).
The derived bracket/product constructions
have been extended to any algebra over binary
quadratic operad in \cite{U1}.\\
\indent
We consider $n$-fold derived brackets
composed of Lie brackets:
$$
[x_{1},...,x_{n}]_{d}:=
(\pm)[[...[[dx_{1},x_{2}],x_{3}]...],x_{n}].
$$
Such higher brackets
were studied by several authors in various contexts
(cf. Akman (1996) \cite{Ak},
Vallejo (2001) \cite{Va},
Roytenberg (2002) \cite{Roy1}
and Voronov (2005) \cite{Vo}).
Koszul's original type higher derived brackets,
which are denoted by $\Phi$,
are defined on super commutative algebras by
$$
\Phi^{n}_{\Delta}(a_{1},...,a_{n}):=
[[...[d\hat{a}_{1},\hat{a}_{2}],...],\hat{a}_{n}](1)
$$
where $d:=[\Delta,-]$ and where
$\Delta$ is a certain differential operator
like $\Delta_{BV}$ above.
The higher brackets $\Phi$ are used to study
{\em higher order} differential operators (cf. \cite{Ak},\cite{Va}).
In \cite{U2}, the author studied a higher derived bracket
construction in the category of {\em Loday} algebras.
We briefly describe the result in \cite{U2}.
Let $(V,\delta_{0})$ be a dg Loday algebra and let
$$
d:=\delta_{0}+t\delta_{1}+t^{2}\delta_{2}+\cdot\cdot\cdot
$$
be a formal deformation of $\delta_{0}$, where $dd=0$.
Define a higher derived bracket on the Loday algebra by
$$
l_{n}(x_{1},...,x_{n}):=
(\pm)[[...[[\delta_{n-1}x_{1},x_{2}],x_{3}]...],x_{n}],
$$
where the binary bracket $[,]$ is a Loday bracket.
It was shown that
the collection of the higher derived brackets
$(l_{1},l_{2},l_{3},...)$
provides a strong homotopy (shortly, sh)
Loday algebra structure
(also called Loday $\infty$-algebra or sh Leibniz algebra
or Leibniz $\infty$-algebra).
If each $l_{n\ge 2}$ is skewsymmetric,
then the sh Loday algebra is an sh Lie ($L_{\infty}$-)algebra.
This proposition is a homotopy version of the binary
derived bracket construction in \cite{K}.
\medskip\\
\indent
The first aim of this note is to study
a higher version of the derived product construction.
Let $(A,\delta_{0})$ be a dg associative algebra
and $d=\sum_{i\ge 0}t^{i}\delta_{i}$ be
a formal deformation of $\delta_{0}$.
Define a higher derived products by
$$
m_{n}(a_{1},...,a_{n}):=
(\pm)\big(\delta_{n-1}a_{1}\big)*a_{2}*\cdot\cdot\cdot*a_{n}.
$$
We show in Theorem \ref{main} below
that the system with the higher derived products
$(A,m_{1},m_{2},...)$
becomes an sh associative algebra (or $A_{\infty}$-algebra).\\
\indent
The second aim of this note is to unify
the higher derived bracket/product constructions.
To complete this task we recall {\em Leibniz pairs}.
The notion of Leibniz pair was introduced by
Flato-Gerstenhaber-Voronov \cite{FGV},
motivated by the study of deformation quantization.
The Leibniz pairs are defined to be
the pairs of Lie and associative algebras $(\g,A)$
equipped with derivation representations
$rep:\g\to\Der(A)$. The representation
satisfies the following two derivation relations:
\begin{eqnarray}
\label{dact1} [x,[a,b]]&=&[[x,a],b]+[a,[x,b]],\\
\label{dact2} [x,[y,a]]&=&[[x,y],a]+[y,[x,a]],
\end{eqnarray}
where $x,y\in\g$, $a,b\in A$ and where $[x,a]$
is the derivation action of $L$ on $A$
and $[a,b]$ is the associative
multiplication on $A$, i.e., $[a,[b,c]]=[[a,b],c]$.
We recall two typical examples of Leibniz pairs.\\
\noindent
a) The self pair of a Poisson algebra $P$,
$(P,P)$, is obviously a Leibniz pair.\\
b) Let $\g\to M$ be a Lie algebroid
over a smooth manifold $M$.
Then the pair $(\Gamma \g,C^{\infty}(M))$
is a Leibniz pair, where $\Gamma \g$ is the space
of sections of $\g$.
\medskip\\
\indent
We consider the pairs of Loday algebras
and associative algebras satisfying (\ref{dact1})
and (\ref{dact2}).
We call such pairs the {\em Loday pairs}.
There exists interesting examples of Loday pairs,
which are regarded as noncommutative analogues of Examples a),b).\\
a-1) It is known that
a Poisson manifold is a classical solution of 
a master equation
associated with 2-dimensional topological field theory
(cf. Cattaneo-Felder \cite{CF1,CF2}).
In the 3-dimensional cases,
the classical solutions are known as {\em Courant algebroids}
(Ikeda \cite{IK1,IK2,IK3}, see also \cite{Roy2}).
A Courant algebroid is a vector bundle
$E\to M$ of which the space of sections is a Loday algebra
satisfying some axioms.
When $E$ is a Courant algebroid,
the pair $(\Gamma E,C^{\infty}(M))$ is a Loday pair.\\
b-1) Let $\L\to M$ be a vector bundle over $M$
with a bundle map $\rho:\L\to TM$.
$\L$ is called a Leibniz algebroid
(Ibanez and collaborators \cite{I}),
if the space of sections $\Gamma \L$ has
a Loday bracket satisfying
$[X_{1},fX_{2}]=f[X_{1},X_{2}]+\rho(X_{1})(f)X_{2}$,
where $X_{1},X_{2}\in\Gamma\L$ and $f\in C^{\infty}(M)$.
When $\L$ is a Leibniz algebroid,
the pair $(\Gamma\L,C^{\infty}(M))$ is a Loday pair.\\
\indent
The place of Loday pairs among other objects
may be illustrated by the following table.
\begin{center}
\begin{tabular}{|c|c|c|c|} 
\hline
$\g\backslash A$& commutative & noncommutative & dimension \\ \hline
commutative & Poisson algebras & classical Leibniz pairs & 2 \\
  & Lie algebroids &  & \\ \hline
noncommutative & Courant algebroids & Loday pairs & 3\\ 
 & Leibniz algebroids & & \\ 
\hline
\end{tabular} 
\end{center}
Loday pairs are noncommutative analogues of
Courant/Leibniz algebroids.
\medskip\\
\indent
We introduce a coalgebra description of Loday pairs, and then
study higher derived bracket construction
in the category Loday pairs (see Section 4).
The Leibniz pairs up to homotopy
which are called sh Leibniz pairs
are studied by Kajiura-Stasheff \cite{KS1,KS2}
and by Hoefel \cite{H}
in the context of open-closed string field theory.
We introduce a new type of homotopy algebra,
{\em sh Loday pair}, which is considered as
a noncommutative analogue of Leibniz pair.
We show that sh associative/Loday algebras
are both subalgebras of sh Loday pairs.
The higher derived brackets in the catrogy
of Loday pairs are defined by
$$
n_{i+j}(x_{1},...,x_{i},a_{1},...,a_{j}):=
(\pm)[[...[[[...[\delta_{i+j-1}x_{1},x_{2}],...],
x_{i}],a_{1}],...],a_{j}],
$$
where $x_{\cdot}\in L$, $a_{\cdot}\in A$
and $[,]$ is a multiplication on a Loday pair.
The higher derived brackets and the higher derived products
are both subsystem of $\{n_{i+j}\}$.
The second main result of this note is as follows.
Let $(L,A,\delta_{0})$ be a Loday pair $(L,A)$
with differential $\delta_{0}$
and let $d:=\sum_{i\ge 0}t^{i}\delta_{i}$ be
a deformation of $\delta_{0}$.
We prove in Proposition \ref{lasprop} that
the system with the unified higher derived brackets
$(n_{1},n_{2},n_{3},...)$ is an sh Loday pair.
\medskip\\
\noindent
\textbf{Acknowledgements}.
I would like to say thank to professors Jim Stasheff
and Akira Yoshioka for their kind advices.
\medskip\\
\noindent
\textbf{Notations and Assumptions}.
In the following,
we assume that the characteristic
of the ground field $\mathbb{K}$
is zero and that a tensor product
is defined over the field, $\ot:=\ot_{\mathbb{K}}$.
We follow the Koszul sign convention.
For instance, a linear map $f\ot g:V\ot V\to V\ot V$
satisfies, for any $v_{1}\ot v_{2}\in V\ot V$,
$$
(f\ot g)(v_{1}\ot v_{2})=(-1)^{|g||v_{1}|}f(v_{1})\ot g(v_{2}),
$$
where $|g|$ and $|v_{1}|$ are degrees of $g$ and $v_{1}$.
We will use a degree shifting operator
$s$ (resp. $s^{-1}$) with degree $+1$ (resp. $-1$).
The shifting operators satisfy
$s\ot s=(s\ot 1)(1\ot s)
=-(1\ot s)(s\ot 1)$.
We denote by $(-1)^{o}$
the sign $(-1)^{|o|}$
without miss reading.
%%%%%%%%%%%%%%%%%%%%%%%
\section{Preliminaries}
%%%%%%%%%%%%%%%%%%%%%%%
We consider the tensor space
over a graded vector space $V$:
$$
\bar{T}V:=V\oplus V^{\ot 2}\oplus\cdot\cdot\cdot.
$$
The space $\bar{T}V$ has an associative coalgebra
structure, $\Delta:\bar{T}V\to \bar{T}V\ot\bar{T}V$, defined by
$\Delta(V):=0$ and
\begin{equation}\label{asscomul}
\Delta(v_{1},...,v_{n}):=
\sum^{n}_{i=1}(v_{1},...,v_{i})\ot(v_{i+1},...,v_{n}),
\end{equation}
where $v_{i}\in V$.
Then $(\bar{T}V,\Delta)$ becomes a cofree coalgebra
in the category of {\em nilpotent} coalgebras.
Let $\Coder(\bar{T}V)$ be the space of coderivations, i.e.,
$D^{c}\in\Coder(\bar{T}V)$ satisfies
the coderivation rule:
$$
(D^{c}\ot 1)\Delta+(1\ot D^{c})\Delta=\Delta D^{c}.
$$
It is well-known that $\Coder(\bar{T}V)$ is identified
with the space of the endomorphisms on $V$:
\begin{equation}\label{keyiso}
\Hom(\bar{T}V,V)\cong\Coder(\bar{T}V).
\end{equation}
We recall an explicit formula of the isomorphism.
For a given $i$-ary endomorphism $f:V^{\ot i}\to V$,
we define a coderivation $f^{c}$ by
$f^{c}(V^{n<i}):=0$ and
$$
f^{c}(v_{1},...,v_{n\ge i}):=
\sum^{n-i}_{k=0}(-1)^{f(v_{1}+\cdot\cdot\cdot+v_{k})}
(v_{1},...,v_{k},f(v_{k+1},...,v_{k+i}),v_{k+i+1},...,v_{n}).
$$
The inverse of the mapping
$f\mapsto f^{c}$ is the restriction
(so-called corestriction).\\
\indent
The space $\Coder(\bar{T}V)$ has a canonical Lie bracket of graded
commutator. Therefore $\Hom(\bar{T}V,V)$ has a Lie bracket
which is induced by the isomorphism (\ref{keyiso}).
The induced Lie bracket on $\Hom(\bar{T}V,V)$,
which is denoted by $\{f,g\}$,
is well-known as a {\em Gerstenhaber bracket} on a Hochschild complex.
If $sV$ (the shifted space of $V$) is an associative algebra,
then $\Hom(\bar{T}V,V)$ becomes a Hochschild complex:
$$
\cdot\cdot\cdot\overset{b}{\to}\Hom(V^{\ot n},V)
\overset{b}{\to}\Hom(V^{\ot n+1},V)\overset{b}{\to}\cdot\cdot\cdot.
$$
The coboundary map $b$ is induced by the associative
structure on $sV$ (see Remark \ref{defb}).\\
\indent
If $f^{c}$, $g^{c}$ are coderivations associated with
$i$-ary, $j$-ary endomorphisms, respectively,
then the Lie bracket $[f^{c},g^{c}]$ is
the coderivation associated with
the Gerstenhaber bracket of $f$ and $g$, i.e.,
$$
\{f,g\}^{c}=[f^{c},g^{c}],
$$
where $\{f,g\}$ is an $(i+j-1)$-ary endomorphism.\\
\indent
In the following, we identify $\Coder(\bar{T}V)$ with $\Hom(\bar{T}V,V)$.
Hence we omit the subscript ``$c$" from $f^{c}$
without miss reading.
\begin{definition}
Let $sV$ be the shifted space equipped with
a collection of $i(\ge 1)$-ary endomorphisms,
$m_{i}:(sV)^{\ot i}\to sV$.
We assume that the degree of $m_{i}$ is $2-i$
for each $i$.
We set the shifted map:
$$
\pa_{i}:=s^{-1}\c m_{i}\c(\overbrace{s\ot\cdot\cdot\cdot\ot s}^{i}).
$$
This is an element in $\Hom(\bar{T}V,V)$
or in $\Coder(\bar{T}V)$ up to the identification.
We define a coderivation by
$$
\pa:=\pa_{1}+\pa_{2}+\cdot\cdot\cdot.
$$
The system $(sV,m_{1},m_{2},...)$
is called a strong homotopy (sh) associative algebra,
or sometimes called an $A_{\infty}$-algebra,
if $\pa$ is square zero, or equivalently,
$$
\frac{1}{2}[\pa,\pa]=0.
$$ 
\end{definition}
\begin{remark}\label{defb}
The usual associative algebra can be seen
as a special sh associative algebra such that 
$\pa_{n\neq 2}=0$. In such a case,
we put $b(-):=[\pa_{2},-]$. Then $b$ becomes
the coboundary map of the Hochschild complex.
\end{remark}
%%%%%%%%%%%%%%%%%%%%%%%%%%
\section{Associative cases}
%%%%%%%%%%%%%%%%%%%%%%%%%%%

\subsection{Derived products}

Let $(A,*,\delta_{0})$ be a differential graded (dg)
associative algebra.
We consider a deformation of $\delta_{0}$:
$$
d:=\delta_{0}+t\delta_{1}+t^{2}\delta_{2}+\cdot\cdot\cdot.
$$
The deformation $d$ is a square zero derivation on $A[[t]]$.
The square zero condition of $d$ is equivalent to
the following condition.
\begin{equation}\label{dd0}
\sum_{i+j=Const}\delta_{i}\delta_{j}=0.
\end{equation}
We define the higher derived products on $sA$ by
$$
m_{i}:=(-1)^{\frac{(i-1)(i-2)}{2}}
s\c M_{i}\c
(\overbrace{s^{-1}\ot\cdot\cdot\cdot\ot s^{-1}}^{i})
(s\delta_{i-1}s^{-1}\ot
\overbrace{1\ot\cdot\cdot\cdot\ot 1}^{i-1}),
$$
where
$M_{i}(a_{1},...,a_{i}):=a_{1}*a_{2}*\cdot\cdot\cdot*a_{i}$
for any $a_{1},...,a_{i}\in A$.
By a direct computation, we have
$$
m_{i}(sa_{1},...,sa_{i})=
(\pm)s\Big(\delta_{i-1}a_{1}*a_{2}*\cdot\cdot\cdot*a_{i}\Big),
$$
where
\begin{eqnarray*}
\pm=\left\{
\begin{array}{ll}
(-1)^{a_{1}+a_{3}+\cdot\cdot\cdot+a_{2n+1}+\cdot\cdot\cdot} & i=\text{even},\\
(-1)^{a_{2}+a_{4}+\cdot\cdot\cdot+a_{2n}+\cdot\cdot\cdot} & i=\text{odd}.
\end{array}
\right.
\end{eqnarray*}
The main theorem of this note is as follows.
\begin{theorem}\label{main}
The system with the higher derived products
$(sA,m_{1},m_{2},...)$ is an sh associative algebra.
\end{theorem}
We need some lemmas in order to show this theorem.
\begin{lemma}
Let $\pa_{i}$ be the coderivation associated with
the higher derived product $m_{i}$.
Then $\pa_{i}$ has the following form.
$$
\pa_{i}=M_{i}\c(\delta_{i-1}\ot\overbrace{1\ot\cdot\cdot\cdot\ot 1}^{i-1}).
$$
\end{lemma}
\begin{proof}
\begin{eqnarray*}
\pa_{i}&:=&s^{-1}\c m_{i}\c(s\ot\cdot\cdot\cdot\ot s)\\
&=&(-1)^{\frac{(i-1)(i-2)}{2}}
M_{i}\c(s^{-1}\ot\cdot\cdot\cdot\ot s^{-1})
(s\delta_{i-1}s^{-1}\ot 1\ot\cdot\cdot\cdot\ot 1)(s\ot\cdot\cdot\cdot\ot s)\\
&=&(-1)^{\frac{(i-1)(i-2)}{2}}
M_{i}\c(s^{-1}\ot\cdot\cdot\cdot\ot s^{-1})
(s\delta_{i-1}\ot s\ot\cdot\cdot\cdot\ot s)\\
&=&(-1)^{\frac{(i-1)(i-2)}{2}}(-1)^{i-1}
M_{i}\c(s^{-1}\ot\cdot\cdot\cdot\ot s^{-1})
(s\ot s\ot\cdot\cdot\cdot\ot s)(\delta_{i-1}\ot 1\ot\cdot\cdot\cdot\ot 1)\\
&=&(-1)^{\frac{(i-1)(i-2)}{2}}(-1)^{i-1}(-1)^{\frac{i(i-1)}{2}}
M_{i}\c(\delta_{i-1}\ot 1\ot\cdot\cdot\cdot\ot 1).
\end{eqnarray*}
\end{proof}
Let $\Der(A)$ be the space of derivations on the algebra $(A,*)$.
For any $D\in\Der(A)$, we define an $i$-ary map by
$$
M_{i}D:=M_{i}\c(D\ot
\overbrace{1\ot\cdot\cdot\cdot\ot 1}^{i-1}),
$$
in particular, $M_{1}D=D$. One can identify $M_{i}D$ with
a coderivation in $\Coder(\bar{T}A)$.
\begin{lemma}\label{pl}
For any $D,D^{\p}\in\Der(A)$ and for any $i,j$,
the Lie bracket of the coderivations,
$[M_{i}D,M_{j}D^{\p}]$, is compatible with
the one of the derivations, $[D,D^{\p}]$, namely,
$$
[M_{i}D,M_{j}D^{\p}]=M_{i+j-1}[D,D^{\p}].
$$
\end{lemma}
\begin{proof}
We assume for the sake of simplicity
that the variables have no degree.
We put $\a=(a_{1},...,a_{i+j-1})\in A^{\ot(i+j-1)}$.
Then we have
\begin{eqnarray*}
M_{i}D\c M_{j}D^{\p}(\a)&=&
\sum^{i-1}_{s=0}Da_{1}*\cdot\cdot\cdot*D^{\p}a_{s+1}*\cdot\cdot\cdot*a_{i+j-1} \\
&=&
D\big(D^{\p}a_{1}*\cdot\cdot\cdot*a_{j}\big)*\cdot\cdot\cdot*a_{i+j-1}+
\sum^{i-1}_{s=1}Da_{1}*\cdot\cdot\cdot*D^{\p}a_{s+1}*\cdot\cdot\cdot*a_{i+j-1}\\
&=&
(DD^{\p}a_{1})*\cdot\cdot\cdot*a_{i+j-1}+
(-1)^{DD^{\p}}\sum^{j}_{t=2}D^{\p}a_{1}*\cdot\cdot\cdot*Da_{t}*\cdot\cdot\cdot*a_{i+j-1}\\
&+&
\sum^{i-1}_{s=1}Da_{1}*\cdot\cdot\cdot*D^{\p}a_{s+1}*\cdot\cdot\cdot*a_{i+j-1}.
\end{eqnarray*}
On the other hand, we have
\begin{eqnarray*}
M_{j}D^{\p}\c M_{i}D(\a)&=&
\sum^{j}_{t=1}D^{\p}a_{1}*\cdot\cdot\cdot*Da_{t}*\cdot\cdot\cdot*a_{i+j-1} \\
&=&
D^{\p}\big(Da_{1}*\cdot\cdot\cdot*a_{i}\big)*\cdot\cdot\cdot*a_{i+j-1}+
\sum^{j}_{t=2}D^{\p}a_{1}*\cdot\cdot\cdot*Da_{t}*\cdot\cdot\cdot*a_{i+j-1}\\
&=&
(D^{\p}Da_{1})*\cdot\cdot\cdot*a_{i+j-1}+
(-1)^{D^{\p}D}\sum^{i-1}_{s=1}Da_{1}*\cdot\cdot\cdot*D^{\p}a_{s+1}*\cdot\cdot\cdot*a_{i+j-1}\\
&+&
\sum^{j}_{t=2}D^{\p}a_{1}*\cdot\cdot\cdot*Da_{t}*\cdot\cdot\cdot*a_{i+j-1}.
\end{eqnarray*}
Hence we obtain
\begin{eqnarray*}
[M_{i}D,M_{j}D^{\p}](\a)&=&
(DD^{\p}a_{1})*\cdot\cdot\cdot*a_{i+j-1}-(-1)^{DD^{\p}}(D^{\p}Da_{1})*\cdot\cdot\cdot*a_{i+j-1}\\
&=&M_{i+j-1}[D,D^{\p}](\a).
\end{eqnarray*}
\end{proof}
We give a proof of Theorem \ref{main}:
\begin{proof}
The higher derived product $m_{i}$
corresponds to the coderivation $\pa_{i}=M_{i}\delta_{i-1}$.
The deformation condition $[d,d]/2=0$
corresponds to the homotopy algebra condition,
$$
\sum_{i+j=Const}[\pa_{i},\pa_{j}]=
\sum_{i+j=Const}[M_{i}\delta_{i-1},M_{j}\delta_{j-1}]=
M_{i+j-1}\sum_{i+j=Const}[\delta_{i-1},\delta_{j-1}]=0.
$$
\end{proof}
We consider the special case of $m_{n\neq 2}=0$,
namely, the case of trivial deformation:
$$
d=t\delta_{1}.
$$
\begin{corollary}
Assume that $m_{n\neq 2}=0$, or equivalently,
$sA$ is the usual associative algebra with
the binary derived product.
Then the collection of $\{M_{i}\Der(A)\}$ is a subcomplex
of the Hochschild complex $\Hom(\bar{T}A,A)$,
where
$$
M_{i}\Der(A):=\langle M_{i}D \ | \ D\in \Der(A)\rangle.
$$
\end{corollary}
\begin{proof}
The coboundary map on $\Hom(\bar{T}A,A)$ is given by
$$
b(-):=[\pa_{2},-]=[M_{2}\delta_{1},-].
$$
Hence we obtain $b(M_{i}D)=M_{i+1}[\delta_{1},D]$.
\end{proof}
%%%%%%%%%%%%%%%%%%%%%%%%%%%%%%%
\subsection{Deformation theory}
%%%%%%%%%%%%%%%%%%%%%%%%%%%%%%%
We discuss a relationship
between deformation theory and sh associative algebras.
The main result of this subsection is Proposition \ref{lastprop} below.
A Loday algebra version of this proposition was shown in \cite{U1}.
\medskip\\
\indent
The deformation of $\delta_{0}$,
$d=\delta_{0}+t\delta_{1}+\cdot\cdot\cdot$,
is considered as a differential on $A[[t]]$
which is an associative algebra of formal series
with coefficients in $A$.
Let $h(t):=th_{1}+t^{2}h_{2}+\cdot\cdot\cdot$ be a derivation
on the associative algebra $A[[t]]$ with degree $|h(t)|:=0$.
We consider the second deformation
$d^{\p}=\sum t^{n}\delta^{\p}_{n}$.
The deformations $d$ and $d^{\p}$
are equivalent, if they are related via the
gauge transformation:
\begin{equation*}
d^{\p}:=exp(X_{h(t)})(d),
\end{equation*}
where $X_{h(t)}:=[-,h(t)]$.
We denote by $\pa^{\p}=\sum\pa_{n}^{\p}$
the induced sh associative structure associated with $d^{\p}$.
\begin{proposition}\label{lastprop}
If $d$ and $d^{\p}$ are gauge equivalent,
then the sh associative structures $\pa$ and $\pa^{\p}$
are equivalent, namely,
\begin{equation*}
\pa^{\p}=exp(X_{Mh})(\pa),
\end{equation*}
where $Mh$ is a well-defined infinite sum
of coderivations:
$$
Mh:=M_{2}h_{1}+M_{3}h_{2}+\cdot\cdot\cdot
+M_{i+1}h_{i}+\cdot\cdot\cdot,
$$
and the integral of $Mh$,
$$
e^{Mh}:=1+Mh+\frac{1}{2!}(Mh)^{2}+\cdot\cdot\cdot,
$$
is a dg coalgebra isomorphism between
$(\bar{T}A,\pa)$ and $(\bar{T}A,\pa^{\p})$, namely,
(\ref{rlas1}) and (\ref{rlas2}) below hold. 
\begin{eqnarray}
\label{rlas1}\pa^{\p}&=&e^{-Mh}\cdot\pa\cdot e^{Mh},\\
\label{rlas2}\Delta e^{Mh}&=&(e^{Mh}\ot e^{Mh})\Delta.
\end{eqnarray}
\end{proposition}
\begin{proof}
The proof of this proposition
is the same as the one in \cite{U1}.
\end{proof}
In general, an $A_{\infty}$-morphism is defined to be a dg coalgebra
morphism between $(\bar{T}A,\pa)$ and $(\bar{T}A^{\p},\pa^{\p})$.
Hence $e^{Mh}$ is an $A_{\infty}$-isomorphism.
%%%%%%%%%%%%%%%%%%%%%%
\section{Loday pairs}
%%%%%%%%%%%%%%%%%%%%%%
We introduce the concept ``Loday pair"
and study its homotopy algebras.
%%%%%%%%%%%%%%%%%%%%%%%%%%%%%
\subsection{Sh Loday algebras}
%%%%%%%%%%%%%%%%%%%%%%%%%%%%%
We recall sh Loday algebras.
Let $L$ be a graded vector space
and let $sL$ be the shifted space
and let $l_{i}:(sL)^{\ot i}\to sL$ be
a multilinear map with degree $2-i$, for each $i\ge 1$.
\begin{definition}\label{shdef}
(\cite{AP}, and see also \cite{U1})
The system with the multiplications,
$(sL,l_{1},l_{2},...)$,
is called a strong homotopy (sh)
Loday algebra (Loday $\infty$-algebra
or sh Loday algebra or Loday $\infty$-algebra),
if the collection $\{l_{i}\}_{i\ge 1}$
satisfies (\ref{homleib}) below.
\begin{multline}\label{homleib}
\sum_{i+j=Const}
\sum^{i+j-1}_{k=j}\sum_{\sigma}
\chi(\sigma)(-1)^{(k+1-j)(j-1)}(-1)^
{j(sx_{\sigma(1)}+...+sx_{\sigma(k-j)})}\\
l_{i}(sx_{\sigma(1)},...,sx_{\sigma(k-j)},
l_{j}(sx_{\sigma(k+1-j)},...,sx_{\sigma(k-1)},sx_{k}),
sx_{k+1},...,sx_{i+j-1})=0,
\end{multline}
where $(sx_{1},...,sx_{i+j-1})\in sL^{\ot(i+j-1)}$,
$\sigma$ is a $(k-j,j-1)$-unshuffle,
$\chi(\sigma)$ is an anti-Koszul sign,
$\chi(\sigma):=sgn(\sigma)\e(\sigma)$.
\end{definition}
Sh Lie algebras are special examples of sh Loday algebras
such that all $l_{i}$ ($i\ge2$) skewsymmetric.
It is easy to show this claim.
If each $l_{i\ge 2}$ is skewsymmetric, then
\begin{multline*}
l_{i}(sx_{\sigma(1)},...,sx_{\sigma(k-j)},
l_{j}(sx_{\sigma(k+1-j)},...,sx_{\sigma(k-1)},sx_{k}),
sx_{k+1},...,sx_{i+j-1})=\\
(\pm)
l_{i}(l_{j}(sx_{\sigma(k+1-j)},...,sx_{\sigma(k-1)},sx_{k}),
sx_{\sigma(1)},...,sx_{\sigma(k-j)},
sx_{k+1},...,sx_{i+j-1})=\\
(\pm)
l_{i}(l_{j}(sx_{\tau(1)},...,sx_{\tau(j)}),
sx_{\tau(j+1)},...,sx_{\tau(i+j-1)}),
\end{multline*}
where $\tau$ is an unshuffle permutation.
And $\sum^{i+j-1}_{k=j}\sum_{\sigma}$ changes into $\sum_{\tau}$.
Then (\ref{homleib}) becomes,
$$
\sum_{i+j=Const}\sum_{\tau}(\pm)
l_{i}(l_{j}(sx_{\tau(1)},...,sx_{\tau(j)}),
sx_{\tau(j+1)},...,sx_{\tau(i+j-1)})=0.
$$
This is the defining relation of sh Lie algebras.\\
\indent
The cofree nilpotent dual-Loday\footnote{
The word ``dual-" means the Koszul duality (\cite{GK}), i.e.,
the operad of dual-Loday algebras is the Koszul dual of the operad
of Loday algebras.}
coalgebra over $L$ is the tensor space $\bar{T}L$ with
a comultiplication, $\Delta_{L}:\bar{T}L\to\bar{T}L\ot\bar{T}L$,
defined by $\Delta_{L}(L):=0$ and
\begin{equation}\label{deflebco}
\Delta_{L}(x_{1},...,x_{n+1}):=
\sum^{n}_{i=1}\sum_{\sigma}\e(\sigma)
(x_{\sigma(1)},x_{\sigma(2)},...,x_{\sigma(i)})\ot
(x_{\sigma(i+1)},...,x_{\sigma(n)},x_{n+1}),
\end{equation}
where $\e(\sigma)$ is a Koszul sign,
$\sigma$ is an $(i,n-i)$-unshuffle.
Let $\Coder(\bar{T}L)$ be the space of
coderivations on $\bar{T}L$.
By a standard argument, we obtain an isomorphism:
$$
\Coder(\bar{T}L)\cong\Hom(\bar{T}L,L).
$$
We recall an explicit formula of the isomorphism.
Let $f:L^{\ot i}\to L$ be an $i$-ary map.
It is one of the generators in $\Hom(\bar{T}L,L)$.
The coderivation associated with $f$
is defined by $f^{c}(L^{\ot n<i}):=0$ and
\begin{multline}\label{deffc}
f^{c}(x_{1},...,x_{n\ge i}):=
\sum^{n}_{k=i}\sum_{\sigma}
\e(\sigma)(-1)^{f(x_{\sigma(1)}+...+x_{\sigma(k-i)})}\\
(x_{\sigma(1)},...,x_{\sigma(k-i)},
f(x_{\sigma(k+1-i)},...,x_{\sigma(k-1)},x_{k}),x_{k+1},...,x_{n}),
\end{multline}
where $\sigma$ is a $(k-i,i-1)$-unshuffle.
The inverse of $f\mapsto f^{c}$ is the restriction
(so-called corestriction).
The collection of $i$-ary maps, $\{l_{i}\}_{i\ge 1}$,
induces a collection of coderivations on $\bar{T}L$,
$\{\pa_{i}\}_{i\ge 1}$.
We put $\pa_{L}:=\pa_{1}+\pa_{2}+\cdot\cdot\cdot$.
The definition (\ref{homleib}) is equivalent to
the homogenous condition:
$$
\frac{1}{2}[\pa_{L},\pa_{L}]=\pa_{L}\pa_{L}=0.
$$
%%%%%%%%%%%%%%%%%%%%%%%%%%%
\subsection{Regularization}
%%%%%%%%%%%%%%%%%%%%%%%%%%%
For the Leibniz identity,
$[x_{1},[x_{2},x_{3}]]=[[x_{1},x_{2}],x_{3}]+[x_{2},[x_{1},x_{3}]]$,
one can regard the term ``$[x_{2},[x_{1},x_{3}]]$"
as an associative anomaly, that is,
$[x_{1},[x_{2},x_{3}]]-[[x_{1},x_{2}],x_{3}]\equiv 0$
modulus $[x_{2},[x_{1},x_{3}]]$.
We notice that the sh Loday relation (\ref{homleib})
has an sh associative anomaly.
We take out the {\em regular}\footnote{
The word ``regular" is used in the sense of $\sigma=id$,
i.e., the order of variables is regular.}
subterms from (\ref{homleib}):
\begin{multline}\label{homleib2}
\sum_{i+j=Const}
\sum^{i-1}_{a=0}
(-1)^{(a+1)(j-1)}(-1)^
{j(sx_{1}+...+sx_{a})}\\
l_{i}(sx_{1},...,sx_{a},
l_{j}(sx_{a+1},...,sx_{a+j}),
sx_{a+j+1},...,sx_{i+j-1}),
\end{multline}
where $a:=k-j$.
This is the defining relation of sh associative algebras.
Thus the no regular terms can be seen as the obstruction
for sh associativity.
In the same way,
we take out the regular subterms from (\ref{deflebco}).
Then it has the same form as the comultiplication
of the associative coalgebra (cf. (\ref{asscomul})):
\begin{equation}\label{deflebco2}
\Delta_{L}(x_{1},...,x_{n+1})\overset{\text{regular}}{\sim}
\sum^{n}_{i=1}
(x_{1},x_{2},...,x_{i})\ot
(x_{i+1},...,x_{n},x_{n+1}).
\end{equation}
As observed above, the associative world
is the regular subsystem in the Loday/Leibniz world.
Along this picture, we try to unify the
sh Loday/associative algebras.
%%%%%%%%%%%%%%%%%%%%%%%%
\subsection{Unification}
%%%%%%%%%%%%%%%%%%%%%%%%
We consider a pair of graded vector spaces $(L,A)$.
Set a tensor space:
$$
LA:=\sum_{n\ge 1}\sum_{i+j=n}L^{\ot i}\ot A^{\ot j}.
$$
Define a comultiplication $\Delta$ on $LA$ by the same manner
as (\ref{deflebco}).
For instance,
\begin{multline}\label{exla}
\Delta(x,a_{1},a_{2},a_{3})=
x\ot(a_{1},a_{2},a_{3})\pm a_{1}\ot(x,a_{2},a_{3})\pm
a_{2}\ot(x,a_{1},a_{3})\pm\\
(x,a_{1})\ot(a_{2},a_{3})\pm
(x,a_{2})\ot(a_{1},a_{3})\pm(a_{1},a_{2})\ot(x_{1},a_{3})\pm\\
(x_{1},a_{1},a_{2})\ot a_{3},
\end{multline}
where $x\in L$ and $a_{1},a_{2},a_{3}\in A$.
The space of coderivations on $(LA,\Delta)$,
$\Coder(LA,\Delta)$,
is identified with a subspace of $\Hom(LA,L\oplus A)$.
This identification is defined by the same rule as (\ref{deffc}).
For instance, if a binary map
$f:\sum_{i+j=2}L^{\ot i}\ot A^{\ot j}\to L\oplus A$
corresponds to a coderivation, then it satisfies
\begin{multline}\label{exfc}
f^{c}(x,a_{1},a_{2},a_{3})=(f(x,a_{1}),a_{2},a_{3})\pm\\
(x,f(a_{1},a_{2}),a_{3})\pm
(a_{1},f(x,a_{2}),a_{3})\pm\\
(x,a_{1},f(a_{2},a_{3}))\pm
(x,a_{2},f(a_{1},a_{3}))\pm
(a_{1},a_{2},f(x,a_{3})).
\end{multline}
We notice that $f(x,a_{2})$ and $f(a_{2},a_{3})$ are $A$-valued,
because the elements of $A$ must be put on the right side
of the elements of $L$. By a simple observation, we obtain
\begin{equation}\label{clahom}
\Coder(LA,\Delta)\cong
\Hom(\bar{T}L,L)\oplus
\sum_{i\ge 0,j\ge 1}\Hom(L^{\ot i}\ot A^{\ot j},A).
\end{equation}
\indent
We are regularizing $\Delta$ with respect to the order
of variables of $A$ (we say this operation an $A$-regularization).
For instance, the $A$-regularization of (\ref{exla}) is
\begin{multline*}
\Delta(x,a_{1},a_{2},a_{3})\overset{\text{$A$-regular}}{\sim}
x\ot(a_{1},a_{2},a_{3})\pm a_{1}\ot(x,a_{2},a_{3})\pm\\
(x,a_{1})\ot(a_{2},a_{3})\pm(a_{1},a_{2})\ot(x_{1},a_{3})\pm
(x_{1},a_{1},a_{2})\ot a_{3},
\end{multline*}
where the ordering $a_{1}<a_{2}<a_{3}$ is preserved.
We denote by $reg(\Delta)$ the regularized comultiplication.
It generally has the form:
$$
reg(\Delta)(\x,\a):=
\sum_{i+k}\sum_{\sigma}(\pm)
(x_{\sigma(1)},...,x_{\sigma(i)},a_{1},...,a_{k})\ot
(x_{\sigma(i+1)},...,x_{\sigma(m)},a_{k+1},...,a_{n}),
$$
where $\x=(x_{1},...,x_{m})$, $\a=(a_{1},...,a_{n})$
and $\sigma$ is an $(i,m-i)$-unshuffle.
Remark that the restrictions $reg(\Delta)|_{L}$
and $reg(\Delta)|_{A}$ coincide with the classical
comultiplications above, respectively.
Let $D^{c}$ be a coderivation on the coalgebra $(LA,\Delta)$.
We define the $A$-regularization of $D^{c}$, $reg(D^{c})$,
by regularizing the order of variables of $A$.
For instance, the regularization of (\ref{exfc}) becomes
\begin{multline*}
reg(f^{c})(x,a_{1},a_{2},a_{3})=(f(x,a_{1}),a_{2},a_{3})\pm\\
(a_{1},f(x,a_{2}),a_{3})\pm
(x,f(a_{1},a_{2}),a_{3})\pm\\
(x,a_{1},f(a_{2},a_{3}))\pm
(a_{1},a_{2},f(x,a_{3})).
\end{multline*}
\begin{lemma}
The regularization $reg(D^{c})$ is a coderivation on
$(LA,reg(\Delta))$, that is,
$$
\big(reg(D^{c})\ot 1+1\ot reg(D^{c})\big)reg(\Delta)
=reg(\Delta)reg(D^{c}).
$$
\end{lemma}
\begin{proof}
Apply $reg$ on the both-side of
$\big(D^{c}\ot 1+1\ot D^{c}\big)\Delta
=\Delta D^{c}$.
We have
$reg\Big((D^{c}\ot 1)\Delta\Big)=
reg\Big((D^{c}\ot 1)reg(\Delta)\Big)=(reg(D^{c})\ot 1)reg(\Delta)$
and $reg\Big((1\ot D^{c})\Delta\Big)=
reg\Big((1\ot D^{c})reg(\Delta)\Big)=(1\ot reg(D^{c}))reg(\Delta)$.
Hence we obtain
$$
reg\Big(\big(D^{c}\ot 1+1\ot D^{c}\big)\Delta\Big)
=\big(reg(D^{c})\ot 1+1\ot reg(D^{c})\big)reg(\Delta).
$$
On the other hand, we obtain
$reg(\Delta D^{c})=reg(\Delta reg(D^{c}))=reg(\Delta)reg(D^{c})$.
Therefore we get the identity of the lemma.
\end{proof}
The space of coderivations on the regularized coalgebra
$(LA,reg(\Delta))$ also corresponds to the same
homomorphism space as (\ref{clahom}).
The above lemma implies that the correspondence
is the $A$-regularization of the isomorphism in (\ref{clahom}).
%%%%%%%%%%%%%%%%%%%%%%%%%%%%%%%%%%%%%
\subsection{Unified derived brackets}
%%%%%%%%%%%%%%%%%%%%%%%%%%%%%%%%%%%%%
Let $L$ be a Loday algebra
and let $A$ be an associative algebra.
We assume that the degrees of the multiplications on
$(L,A)$ are both zero (or even).
\begin{definition}
The pair $(L,A)$ equipped with 
a binary multiplication $[,]:L\ot A\to A$
is called a left-Loday pair, or simply, Loday pair,
if it satisfies
\begin{eqnarray}
\label{l1} [x,[y,a]]&=&[[x,y],a]+(-1)^{xy}[y,[x,a]],\\
\label{l2} [x,[a,b]]&=&[[x,a],b]+(-1)^{xa}[a,[x,b]],
\end{eqnarray}
where $x,y\in L$, $a,b\in A$ and
where $[x,y]$ is the Loday bracket on $L$ and
$[a,b]$ is the associative multiplication on $A$, i.e.,
$[a,[b,c]]=[[a,b],c]$ for any $a,b,c\in A$.
\end{definition}
The Loday pairs are algebraizations
of Leibniz algebroids (\cite{I}).
The classical Leibniz pair in \cite{FGV}
is the Lie version of our noncommutative Leibniz pair
(i.e. Loday pair).
We give two geometric examples of Loday pairs.
\begin{example}
(Courant bracket)
Let $M$ be a smooth manifold.
Consider a bundle $\mathcal{T}M:=TM\oplus T^{*}M$
(so-called generalized tangent bundle).
The Courant bracket is defined on
the space of sections of $\mathcal{T}M$,
$\Gamma\mathcal{T}M$, by
$[\xi_{1}+\theta_{1},\xi_{2}+\theta_{2}]
:=[\xi_{1},\xi_{2}]+\L_{\xi_{1}}\theta_{2}-i_{\xi_{2}}d\theta_{1}$,
where $\xi_{1},\xi_{2}\in\Gamma TM$
and $\theta_{1},\theta_{2}\in\Gamma T^{*}M$.
Then the pair $(\Gamma\mathcal{T}M,C^{\infty}(M))$
is a Loday pair.
\end{example}
\begin{example}\label{2ndex}
Let $(M,\pi)$ be a Poisson manifold
equipped with a Poisson structure tensor $\pi$.
The space of multivector fields $\Gamma\bigwedge^{\cdot}TM$
becomes a graded Poisson algebra of type $(-1,0)$,
whose Poisson bracket is known as Schouten-Nijenhuis (SN) bracket.
The Poisson tensor is a solution of
Maurer-Cartan equation $[\pi,\pi]_{SN}=0$.
Since the degree of $\pi$ is $+2$,
$d:=[\pi,-]_{SN}$ becomes a differential with degree $+1$.
This differential is the coboundary operator
of the Poisson cohomology.
We define a Loday bracket by $[X,Y]_{\pi}:=(-1)^{X}[dX,Y]_{SN}$
for any $X,Y\in\Gamma\bigwedge^{\cdot}TM$.
The bracket $[,]_{\pi}$ is the derived bracket of $SN$-bracket
by the Poisson structure.
Then the self pair
$(\Gamma\bigwedge^{\cdot}TM,\Gamma\bigwedge^{\cdot}TM)$
is a Loday pair with multiplications
$[,]_{\pi}$ and $\wedge$.
In particular, the sub-pair $(C^{\infty}(M),C^{\infty}(M))$
is the self pair of the Poisson algebra on $M$.
\end{example}
We get a natural result.
\begin{corollary}
The structure of Loday pair on $(sL,sA)$
is equivalent to a (binary) codifferential
on the regularized coalgebra $(LA,reg(\Delta))$.
\end{corollary}
This corollary leads us to
\begin{definition}
An sh (left) Loday pair is, by definition,
a pair $(sL,sA)$ equipped with a codifferential
on $(LA,reg(\Delta))$.
\end{definition}
We consider a derived bracket construction
in the category of Loday pairs.
\begin{definition}
A derivation on a Loday pair $(L,A)$ is,
by definition, a pair of derivations,
$D=(D_{L},D_{A})$,
$D_{L}\in\Der(L)$ and $D_{A}\in\Der(A)$
satisfying
$$
D[o_{1},o_{2}]=[Do_{1},o_{2}]+(-1)^{Do_{1}}[o_{1},Do_{2}]
$$
for any $o_{1},o_{2}\in(L,A)$.
We assume that the parity of
$D_{L}$ is equal with the one of $D_{A}$.
\end{definition}
\begin{example}
If $(L,A)$ is a Loday pair, then an adjoint action
$[x,-]$, $x\in L$, is a derivation.
\end{example}
A Loday pair $(L,A)$ is called a dg Loday pair,
if it has a differential $\delta$ on $(L,A)$.
It is easy to check that dg Loday pairs
are special sh Loday pairs such that
the higher homotopies vanish.
Given a dg Loday pair $(L,A,\delta)$,
define derived brackets by
\begin{eqnarray*}
\ [sx,sy]_{d}&:=&(-1)^{x}s[\delta x,y],\\
\ [sx,sa]_{d}&:=&(-1)^{x}s[\delta x,a],\\
\ [sa,sb]_{d}&:=&(-1)^{a}s[\delta a,b].
\end{eqnarray*}
Then the triple of the derived brackets provides
a new structure of Loday pair on $(sL,sA)$.
\begin{example}
In Example \ref{2ndex},
if $\pi^{\p}$ is a second Poisson tensor 
which is compatible with $\pi$, i.e.,
$[\pi,\pi^{\p}]_{SN}=0$,
then $\delta:=[\pi^{\p},-]_{SN}$
is a differerential on the Loday pair.
\end{example}
We consider the higher derived bracket construction
for Loday pairs. Let $D$ be a derivation on
a Loday pair $(L,A)$. We put
\begin{equation}\label{nmtr}
N_{k}D(\x,\mathbf{a}):=
[[...[[[...[Dx_{1},x_{2}],...],x_{i}],a_{i+1}],...],a_{i+j}],
\end{equation}
where $\x:=(x_{1},...,x_{i})$,
$\mathbf{a}:=(a_{i+1},...,a_{i+j})$
and $k:=i+j$,
in particular, $N_{k}D(\a)=M_{k}D(\a)$.
\begin{lemma}\label{nnn}
We regard $N_{\cdot}D$ as a coderivation
on $(LA,reg(\Delta))$.
For any derivations $D,D^{\p}$ on $(L,A)$
and for any $k,l\ge 1$,
$$
[N_{k}D,N_{l}D^{\p}]=N_{k+l-1}[D,D^{\p}].
$$
\end{lemma}
We will give a proof of this lemma
in the end of this section.
The main result of this section is as follows.
\begin{proposition}\label{lasprop}
Let $(L,A,\delta_{0})$ be a dg Loday pair.
We consider a deformation of $\delta_{0}$,
$d:=\sum_{i\ge 0}t^{i}\delta_{i}$.
For each $k\ge 1$, define a coderivation by
$$
\pa_{k}:=N_{k}\delta_{k-1}.
$$
Then $\pa:=\sum_{k}\pa_{k}$ is a structure of sh Loday pair.\\
\indent
The multilinear map
$N_{k}\delta_{k-1}(\x,\mathbf{a})$ corresponds to
the higher bracket on the sifted pair $(sL,sA)$:
$$
n_{k}(sx_{1},...,sx_{i},sa_{i+1},...,sa_{i+j}):=
(\pm)s[[...[[[...[\delta_{i+j-1}x_{1},x_{2}],...],x_{i}],
a_{i+1}],...],a_{i+j}],
$$
where $k:=i+j$, $k\ge 1$ and
\begin{eqnarray*}
\pm:=\left\{
\begin{array}{ll}
(-1)^{o_{1}+o_{3}+\cdot\cdot\cdot+o_{2n+1}+\cdot\cdot\cdot} & i+j=\text{even},\\
(-1)^{o_{2}+o_{4}+\cdot\cdot\cdot+o_{2n}+\cdot\cdot\cdot} & i+j=\text{odd}.
\end{array}
\right.
\end{eqnarray*}
and where $o_{\cdot}\in\{x_{\cdot},a_{\cdot}\}$.
The restrictions $(sL,n|_{sL})$ and $(sA,n|_{sA})$
become an sh Loday algebra and an sh associative algebra,
respectively.
\end{proposition}
We give a proof of Lemma \ref{nnn}.
To show this lemma we use convenient symboles:
\begin{eqnarray*}
\ [x_{1},...,x_{i}]&:=&[[[x_{1},x_{2}],...],x_{i}],\\
\ D\x&:=&[Dx_{1},...,x_{i}],
\end{eqnarray*}
where $\x:=(x_{1},...,x_{i})$.
The pure Loday version of the lemma was shown in \cite{U1}.
We consider the mixed case.
Since the adjoint action
$\what{L}:=[L,-]$ is a derivation on $A$, (\ref{nmtr}) becomes
$$
N_{k}D(\x,\mathbf{a})=(M_{j}\what{D\x})(\a),
$$
We denote by $|\cdot|$ the length of word.
\begin{lemma}
Assume that $|(\x,\a)|:=k+l-1$ and $|\a|\ge 1$.
\begin{equation}\label{mint}
[N_{k}D,N_{l}D^{\p}](\x,\a)
=\sum_{(\x_{1},\x_{2})}M_{|\a|}[\what{D\x_{1}},\what{D^{\p}\x_{2}}](\a),
\end{equation}
where $(\x_{1},\x_{2})$ runs over the unshuffle-permutations
including $(\emptyset,\x)$ and $(\x,\emptyset)$.
\end{lemma}
\begin{proof}
$N_{l}D^{\p}(\x,\a)$ is decomposed into
the pure Loday term (if it exists) and the mixed term:
\begin{eqnarray*}
N_{l}D^{\p}(\x,\a)&=&
\sum_{|\x_{2}|=l}(\x_{1},D^{\p}\x_{2},\x_{3},\a)+
\sum_{|\x_{2}|<l}(\x_{1},\a_{1},D^{\p}(\x_{2},\a_{2}),\a_{3})\\
&=&
\sum_{|\x_{2}|=l}(\x_{1},D^{\p}\x_{2},\x_{3},\a)+
\sum_{|\x_{2}|<l}(\x_{1},(M_{l-|\x_{2}|}\what{D^{\p}\x_{2}})(\a)).
\end{eqnarray*}
We have
\begin{multline*}
N_{k}D\c N_{l}D^{\p}(\x,\a)=\\
\sum_{|\x_{2}|=l}[D\x_{1},D^{\p}\x_{2},\x_{3},\a]+
\sum_{|\x_{2}|<l}
\Big(M_{k-|\x_{1}|}\what{D\x_{1}}\c M_{l-|\x_{2}|}\what{D^{\p}\x_{2}}\Big)(\a),
\end{multline*}
and this gives
\begin{multline*}
[N_{k}D,N_{l}D^{\p}](\x,\a)=\\
\sum_{|\x_{2}|=l}[D\x_{1},D^{\p}\x_{2},\x_{3},\a]-
\sum_{|\x_{2}|=k}[D^{\p}\x_{1},D\x_{2},\x_{3},\a]+\\
\sum_{|\x_{1}|<k,|\x_{2}|<l}
[M_{k-|\x_{1}|}\what{D\x_{1}},M_{l-|\x_{2}|}\what{D^{\p}\x_{2}}](\a).
\end{multline*}
The third term has the desired formula:
$$
\sum_{\text{$|\x_{1}|<k$ and $|\x_{2}|<l$}}
M_{|\a|}[\what{D\x_{1}},\what{D^{\p}\x_{2}}](\a).
$$
It is not difficult to show that (see Appendix below)
\begin{equation}\label{pape}
\sum_{|\x_{2}|=l}[D\x_{1},D^{\p}\x_{2},\x_{3},\a]-
\sum_{|\x_{2}|=k}[D^{\p}\x_{1},D\x_{2},\x_{3},\a]=
\sum_{\text{$|\x_{1}|\ge k$ or $|\x_{2}|\ge l$}}
M_{|\a|}[\what{D\x_{1}},\what{D^{\p}\x_{2}}](\a).
\end{equation}
Hence we obtain (\ref{mint}).
\end{proof}
If $\x=x$ or $|\x|=1$, then
$(\x_{1},\x_{2})\in\{(\emptyset,x),(x,\emptyset)\}$
and
\begin{eqnarray*}
\sum_{(\x_{1},\x_{2})}M_{|\a|}[\what{D\x_{1}},\what{D^{\p}\x_{2}}]&=&
M_{|\a|}[D,\what{D^{\p}x}]+M_{|\a|}[\what{Dx},D^{\p}]\\
&=&M_{|\a|}\what{[D,D^{\p}]x}.
\end{eqnarray*}
By using induction for the length of $\x$,
one can easily prove that
$$
\sum_{(\x_{1},\x_{2})}M_{|\a|}[\what{D\x_{1}},\what{D^{\p}\x_{2}}]=
M_{|\a|}\what{[D,D^{\p}]\x}.
$$
Hence we obtain the identity of Lemma \ref{nnn}:
$$
[N_{k}D,N_{l}D^{\p}](\x,\a)=(M_{|\a|}\what{[D,D^{\p}]\x})(\a)
=N_{l+k-1}[D,D^{\p}](\x,\a).
$$
\textbf{Appendix}. We show (\ref{pape}).
For any $A,B\in L$ and for any $\y:=(y_{1},...,y_{n})\in L^{\ot n}$,
by the Leibniz rule, we have
$$
[[A,B,\y],-]=\sum_{(\y_{1},\y_{2})}[[[A,\y_{1}],[B,\y_{2}]],-]
$$
where $(\y_{1},\y_{2})$ are the unshuffle permutations of $\y$
including $(\emptyset,\y)$ and $(\y,\emptyset)$.
Now, replace $A\to D\x_{1}$, $B\to D^{\p}\x_{2}$ and $\y\to \x_{3}$.
Then
\begin{eqnarray*}
\sum_{|\x_{2}|=l}[D\x_{1},D^{\p}\x_{2},\x_{3},\a]
&=&\sum_{|\x_{2}|=l}[[D\x_{1},\y_{1}],[D^{\p}\x_{2},\y_{2}],\a] \\
&=&\sum_{|\x_{2}|=l}[D(\x_{1},\y_{1}),D^{\p}(\x_{2},\y_{2}),\a] \\
&=&\sum_{|\x_{2}|\ge l}[[D\x_{1},D^{\p}\x_{2}],\a] \\
&=&\sum_{|\x_{2}|\ge l}M_{|\a|}\what{[D\x_{1},D^{\p}\x_{2}]}(\a)
=\sum_{|\x_{2}|\ge l}M_{|\a|}[\what{D\x_{1}},\what{D^{\p}\x_{2}}](\a)
\end{eqnarray*}
where $\x_{i}:=(\x_{i},\y_{i})$ redefined $i\in\{1,2\}$.
The other term is, by the same manner,
$$
-\sum_{|\x_{2}|=k}[D^{\p}\x_{1},D\x_{2},\x_{3},\a]
=\sum_{|\x_{1}|\ge k}M_{|\a|}[\what{D\x_{1}},\what{D^{\p}\x_{2}}](\a).
$$
This implies (\ref{pape}).

\noindent
Post doctoral.\\
Tokyo University of Science.\\
3-14-1 Shinjyuku Tokyo Japan.\\
e-mail: K\underline{ }Uchino[at]oct.rikadai.jp
\end{document}